\documentclass[a4paper]{article}

\usepackage[english]{babel}
\usepackage[utf8x]{inputenc}
\usepackage[T1]{fontenc}

\usepackage[a4paper,top=3cm,bottom=2cm,left=3cm,right=3cm,marginparwidth=1.75cm]{geometry}

\usepackage{amsmath}
\usepackage{amsthm}
\usepackage{graphicx}
\usepackage[colorinlistoftodos]{todonotes}
\usepackage[colorlinks=true, allcolors=blue]{hyperref}
\usepackage{authblk}
\usepackage{mathtools}

\title{Robust method for finding sparse solutions to linear inverse problems using an L2 regularization}
\author[1]{Gonzalo H. Otazu}
\affil[1]{Cold Spring Harbor Laboratory, Cold Spring Harbor, NY, 11724

ghotazu@gmail.com}
\everymath{\displaystyle}
\newtheorem{theorem}{Theorem}

\begin{document}
\maketitle

\begin{abstract}
We analyzed the performance of a biologically inspired algorithm called the Corrected Projections Algorithm (CPA) when a sparseness constraint is required to unambiguously reconstruct an observed signal using atoms from an overcomplete dictionary. By changing the geometry of the estimation problem, CPA gives an analytical expression for a binary variable that indicates the presence or absence of a dictionary atom using an L2 regularizer. The regularized solution can be implemented using an efficient real-time Kalman-filter type of algorithm.  The smoother L2 regularization of CPA makes it very robust to noise, and CPA outperforms other methods in identifying known atoms in the presence of strong novel atoms in the signal. 
\end{abstract}

\section{Introduction}

Representation of time-varying signals in terms of a sparse set of atoms from an overcomplete dictionary is important in machine learning, and it has been proposed as one of the fundamental computations of  the nervous system. Reconstruction algorithms create an estimate of the observed signal using the appropriate weighted atoms of an overcomplete dictionary. The overcompleteness of the dictionary creates a situation where there are multiple sets of dictionary atoms that reconstruct the signal equally well. Under these conditions, a solution that minimizes the number of dictionary atoms that have a non-zero contribution to the signal is preferred. Directly finding the solution that minimizes the number of atoms used from the dictionary or L0 norm is an NP-complete problem, making it intractable even for moderately sized dictionaries. However, it has been shown that under very general conditions the solution that minimizes the sum of the absolute values of the contributions of the dictionary atoms or L1 norm can be used to identify the sparsest solution\cite{Donoho2006ForSolution}.  Although there is no analytical expression for the minimal L1 norm solution, there are effective algorithms for finding L1 minimal solutions, which has triggered an explosion of interest in the use of the L1 norm for sparse reconstructions. On the other hand, minimization of the sum of the squares of the contributions or L2 norm yields an analytical solution. However, the minimal L2 norm solutions are, in general, not sparse. 
\footnote{This work has been submitted to the IEEE for possible publication. Copyright may be transferred without notice,
after which this version may no longer be accessible.}

In this contribution we show that an algorithm described as a model for sound identification in the mammalian auditory system\cite{Otazu2011AScenes.}, the Corrected Projections Algorithm or CPA, can be used for finding sparse representations of temporally uncorrelated atoms chosen from overcomplete dictionaries. CPA yields an optimization problem that directly infers the presence or absence of all the dictionary atoms by way of a least-squares minimization using an L2 constraint on the estimated parameters. The L2 regularization adds robustness in the presence of strong noise, outperforming standard sparse representation methods in novel situations.

In section \ref{CPA} we present the problem and the algorithm and show how the algorithm identifies the atoms present in the overdetermined case. In section \ref{CPA_with_L2} we show that the algorithm can extended to an overcomplete situation where the number of temporal observations is insufficient for a non-ambiguous reconstruction of the signal and show that the change in the geometry of the problem allows for a reconstruction using the L2 norm. In section \ref{iCPA and performance} we show how the L2 regularized solution can be implanted using a Kalman filter. In section \ref{CPA_performance} we compare the performance of CPA with other sparse representation algorithms. We show that CPA produces weak dense representations for novel atoms, whereas other sparse representation algorithms create sparse amplitude-dependent representations of novel atoms. We also show that CPA outperforms other methods in identifying known atoms in the presence of strong novel atoms.

\section{Corrected Projections Algorithm} \label{CPA}
\subsection{Problem statement}
At each time t, an observation vector $\overrightarrow{y(t)}$ of dimension $N$ by $1$ is produced by the combination of a few atoms of dimension $N$ by $1$ from a dictionary of $M$ possible atoms. That is:

\begin{equation}\label{signal_descript}
\overrightarrow{y(t)}  = \sum_{i=1}^{M} A_i(t)\overrightarrow{B_i}\ \text{,}
\end{equation}
where $\overrightarrow{B_i}$ is the i-th dictionary atom, which is an $N$ by $1$ vector, and $A_i (t)$ is the contribution of the i-th dictionary atom at time $t$. We will represent any $N$ by $1$ vector $X$ using the symbol $\overrightarrow{X}$.
The general assumption is that the observation is generated by a few dictionary atoms that belong to an active set $A$ of $k$ atoms, whereas the other atom's contributions are equal to zero:
\begin{equation}\label{sparse_contrib}
A_i (t)=0, i∉ A,\forall{t}.
\end{equation}
In order to determine the dictionary atoms that contributed to the observed signal, an estimate $\widehat{y(t)}$ of each temporal observation   is calculated by combining the dictionary atoms by way of an estimate $\widehat{A_i (t)}$  of the contributions of each dictionary atom, that is:
\begin{equation}\label{estimate_descript}
\widehat{y(t)}  = \sum_{i=1}^{M} \widehat{A_i(t)}\overrightarrow{B_i}\ \text{.}
\end{equation}
The estimates of the contributions are found by minimizing a cost function defined as the Euclidean distance between the estimate and the observation, that is:
\begin{equation}\label{simple_cost}
Cost(t) =\left|\overrightarrow{y(t)}-\widehat{y(t)}\right|^2 \text{.}
\end{equation}
We will represent the  L2 norm of a column vector $X$  as $\left|X \right|^2=X^TX$ .	In general, for large dictionaries, there might be multiple solutions that minimize the cost function. Therefore, current algorithms complement the cost function with a penalty associated with the value of the estimated contribution of the dictionary atoms, that is:
\begin{equation}\label{cost_with_regular}
Cost(t) =\left|\overrightarrow{y(t)}-\widehat{y(t)}\right|^2+\lambda \sum_{i=1}^{M} |A_i(t)|\ \text{,}
\end{equation}
where $\lambda$ is a parameter that determines the balance between the estimation error and the degree of sparseness of the solution.
Single measurement algorithms such as Matching Pursuit\cite{Mallat1993MatchingDictionaries}  calculate a time-varying estimate $\widehat{A_i(t)}$  of the contribution $A_i(t)$ for each individual observation $\overrightarrow{y(t)}$ while trying to reduce the number of dictionary atoms that make a non-zero contribution.  Conversely, multiple measurement vector (MMV) versions of these algorithms \cite{Cotter2005SparseVectors} take into account all the temporal observations available of the signal and use the sparseness constraint to minimize the number of dictionary atoms that make any contribution at any point in time, that is:
\begin{equation}\label{cost_with_regular_MMV}
Cost =\sum_{t=1}^{T} \left|\overrightarrow{y(t)}-\widehat{y(t)}\right|^2+\lambda \sum_{t=1}^{T} \sum_{i=1}^{M} |A_i(t)|^p\ \text{,}
\end{equation}
where $p≤1$. The use of multiple observations allows for better estimates compared to using individual observations. 
\subsection{Corrected Projections Algorithm}
CPA is similar to multiple measurement vector algorithms in that it uses the information from multiple observations to identify a time-invariant binary variable for each dictionary atom that indicates whether or not the corresponding atom was present for any of the temporal observations of the signal. These binary variables, called the presence parameters, can be written as a column vector:
\begin{equation}\label{def_theta}
 \Theta = \begin{pmatrix}
  \theta_1  \\
  \theta_2  \\
  \vdots \\
 \theta_M 
 \end{pmatrix} \text{.}
\end{equation}
In order to create the estimate of the observed signal, CPA combines a time-varying rough estimate  $\widehat{A_i (t)}$  of the contribution of an individual dictionary atom corrected by the corresponding time-invariant presence parameter $\theta_i$, that is:
\begin{equation}\label{estimate_descript_cpa}
\widehat{y(t)}  = \sum_{i=1}^{M} \theta_i\widehat{A_i(t)}\overrightarrow{B_i}\ \text{.}
\end{equation}
The rough estimate at time $t$ of the contribution of a given dictionary atom is given by the scalar product between the current observation and the dictionary atom, that is:
\begin{equation}\label{cpa_rough_estimate}
\widehat{A_i(t)}={\overrightarrow{B_i} \cdot \overrightarrow{y(t)}}\text{.}
\end{equation}
 We can define an $N$ by $M$ projection matrix $\phi(t)$ at time $t$, where each column corresponds to a dictionary atom weighted by the rough estimate of its contribution to the signal $\overrightarrow{y(t)}$:

\begin{equation}\label{projection_matrix}
\phi(t) = 
 \begin{pmatrix}
  B_{1,1}(\overrightarrow{B_1} \cdot \overrightarrow{y(t)}) &\cdots & B_{1,M}(\overrightarrow{B_M} \cdot \overrightarrow{y(t)}) \\
  B_{2,1}(\overrightarrow{B_1} \cdot \overrightarrow{y(t)}) &\cdots & B_{2,M}(\overrightarrow{B_M} \cdot \overrightarrow{y(t)}) \\
  \vdots    & \ddots & \vdots  \\
  B_{N,1}(\overrightarrow{B_1} \cdot \overrightarrow{y(t)})  & \cdots & B_{N,M}(\overrightarrow{B_M} \cdot \overrightarrow{y(t)}) 
 \end{pmatrix}
\end{equation}
where $B_{j,i}$ is the j-th component of the i-th dictionary atom.
We can arrange all the $T$ available observations $\overrightarrow{y(t)}$ into a long $T xN$ by $1$ vector $Y$:
\begin{equation}\label{observ_vector}
Y=\begin{pmatrix}
\overrightarrow{y(1)} \\
\overrightarrow{y(2)} \\
\vdots  \\
\overrightarrow{y(T)} \\
\end{pmatrix}\text{.}
\end{equation}
We can also arrange all the $T$ projection matrices $\phi(t)$ (one for each temporal observation) into one large $TxN$ by $M$ matrix $\Phi$:
\begin{equation}\label{phi_defini}
\Phi=\begin{pmatrix}
\phi(1) \\
\phi(2) \\
\vdots  \\
\phi(T) \\
\end{pmatrix}\text{.}
\end{equation}
Using this notation, we can write the CPA estimate $\widehat{Y}$  of the observation vector $Y$ as:
\begin{equation}\label{Y_definition}
\widehat{Y}=\Phi\Theta \text{.}
\end{equation}
CPA estimates the presence parameter vector $\Theta$ as the one that minimizes the square error between the observation vector $Y$ and its estimate $\widehat{Y}$. The presence parameter vector is given by:
\begin{equation}\label{standard_CPA}
\Theta=(\Phi{\Phi}^T)^{-1}\Phi^TY \text{.}
\end{equation}
Here we provide a simple proof that the solution of equation \ref{standard_CPA} identifies the atoms present in the signal.
\begin{theorem}
The solution for the presence parameters given by equation \ref{standard_CPA} is $\theta_i=1$ for the atoms that are present and $\theta_i=0$  for the atoms that are not present as long as: 
\begin{enumerate}
\item the dictionary atoms that are simultaneously present in the signal are orthogonal to each other, and
\item the $M$ by $M$ matrix $\Phi{\Phi}^T$ is invertible.
\end{enumerate}
\end{theorem}
\begin{proof}
We will assume that the observed signal $\overrightarrow{y(t)}$ originates from an active set  $A$ of $k$ dictionary atoms indexed by $j(1),\dots,j(k)$ that are mutually orthogonal, that is, the observations are given by:
\begin{equation}\label{definition_proof}
\overrightarrow{y(t)}=\sum_{l=1}^{k} A_{j(l)}(t)\overrightarrow{B_{j(l)}}\
\end{equation}
where the atoms present obey:
\[ {\overrightarrow{B_{j(n)  }} \cdot \overrightarrow{B_{j(m)}}} =
  \begin{cases}
    1       & \quad \text{if } n=m \\
    0       & \quad \text{if } n\not=m \text{.} \\ 
  \end{cases}
\] 
 If we replace this signal into equations \ref{estimate_descript_cpa}  and \ref{cpa_rough_estimate} we obtain:
\begin{equation}\label{replaced_sparse_y}
\widehat{y(t)}  = \sum_{i=1}^{M} \theta_i\left(\overrightarrow{y(t)} \cdot \overrightarrow{B_i}\right)\overrightarrow{B_i}\ 
= \sum_{i=1}^{M} \theta_i\left(\sum_{l=1}^{k} A_{j(l)}(t)\overrightarrow{B_{j(l)}}\\\cdot \overrightarrow{B_i}\right)\overrightarrow{B_i}\ \text{.}
\end{equation}
Using the orthogonality condition for the atoms in the active set, we can simplify our estimate as:
\begin{equation}\label{simple_sparse_y}
\widehat{y(t)}  = \sum_{i\in A}^{} \theta_i A_i(t) \overrightarrow{B_i} +\sum_{j\notin A}^{} \theta_j \left( \sum_{i\in A}^{} A_i(t) \left(\overrightarrow{B_i}\cdot \overrightarrow{B_j}\right)  \right) \overrightarrow{B_j}\text{.}
\end{equation}
If we replace the following solution for the presence parameters 
\begin{equation}
\begin{cases}
    \theta_i=1       & \quad \text{for }i \in{A} \\
    \theta_i=0       & \quad \text{for }i \notin{A} \\
  \end{cases}
\end{equation}
into equation \ref{simple_sparse_y}, we obtain:
\begin{equation}\label{estimate_equal_observ}
\widehat{y(t)}  =\sum_{i\in A}^{}  A_i(t) \overrightarrow{B_i}\text{.}
\end{equation}
With this choice of the presence parameters, our estimate $\widehat{y(t)}$ and the actual observation $\overrightarrow{y(t)}$ are identical for all temporal observations. This would result in the mean square error having a value of zero. Therefore, there are no other solutions that could produce a smaller value of the squared error. The solution is also unique because we have assumed that the matrix $\Phi{\Phi}^T$ is invertible. Notice that this solution for the presence parameters is independent of the contribution of an atom to the observed signal, being either 1 or 0. This contrasts to algorithms that directly determine the amplitude of the contribution $A_i(t)$,  where the estimated variables would be larger for larger contributions.
\end{proof}
The theorem depends on the orthogonality of the atoms present in the observed signal.  Although this condition seems restrictive, the restrictive isometry property \cite{Baraniuk2008AMatrices} (RIP) states that any set of $k$ atoms of a random dictionary of size $M$ would approximate orthogonality, as long as the dimensions of the dictionary are:
\begin{equation} \label{RIP}
  N ≥ k \log{\left(M/k\right)} \text{.}
\end{equation}
Given that the maximum size of the dictionary $M$ grows exponentially with the number of dimensions $N$ of the signal, CPA can handle large dictionaries with a moderately sized number of dimensions $N$.
\section {CPA finds a sparse solution using the L2 regularization for the presence parameters}\label{CPA_with_L2}
In order to identify the dictionary atoms that are present, CPA requires that the matrix $\Phi{\Phi}^T$  has an inverse. For conditions where $(T*N)< M$, this condition would not be satisfied. Therefore, CPA needs to use regularization to find a sparse solution. Surprisingly, CPA does not need to use the L1 regularization but can use the L2 regularization or Tikhonov regularization to identify sparse solutions. The modified cost function is:
\begin{equation}\label{cost_with_L2_CPA}
Cost =\sum_{t=1}^{T} \left|\overrightarrow{y(t)}-\widehat{y(t)}\right|^2+\lambda \sum_{i=1}^{M} {\theta_i}^2 \text{,}
\end{equation}
where the CPA estimate $\widehat{y(t)}$ is described by equations \ref{estimate_descript_cpa} and \ref{cpa_rough_estimate}.  We can also write this cost function as:
\begin{equation}\label{cost_with_L2_CPA_matrix}
Cost = \left(Y-\Phi\Theta\right)^T\left(Y-\Phi\Theta\right)+\lambda\Theta^T \Theta\text{.}
\end{equation}
In contrast to the minimization problems  \ref{cost_with_regular} and \ref{cost_with_regular_MMV} that use the L1 norm, there is an analytical solution for this minimization problem, which is given by:
\begin{equation}\label{regularized_CPA}
\Theta=(\Phi{\Phi}^T +\lambda I )^{-1}\Phi^TY \text{,}
\end{equation}
where $I$ is the $M$ by $M$ identity matrix.
The L2 regularization does not provide a sparse representation when it is used to directly determine the contributions $A_i(t)$ of the dictionary atoms, that is, when the cost function is:
\begin{equation}\label{L2_cost_with_regular_MMV}
Cost =\sum_{t=1}^{T} \left|\overrightarrow{y(t)}-\widehat{y(t)}\right|^2+\lambda \sum_{t=1}^{T} \sum_{i=1}^{M} \left(A_i(t)\right)^2\ \text{.}
\end{equation}
Therefore, in general, the L1 norm has to be used for determining sparse representations.
\begin{figure}
\centering
\includegraphics[width=0.6\textwidth]{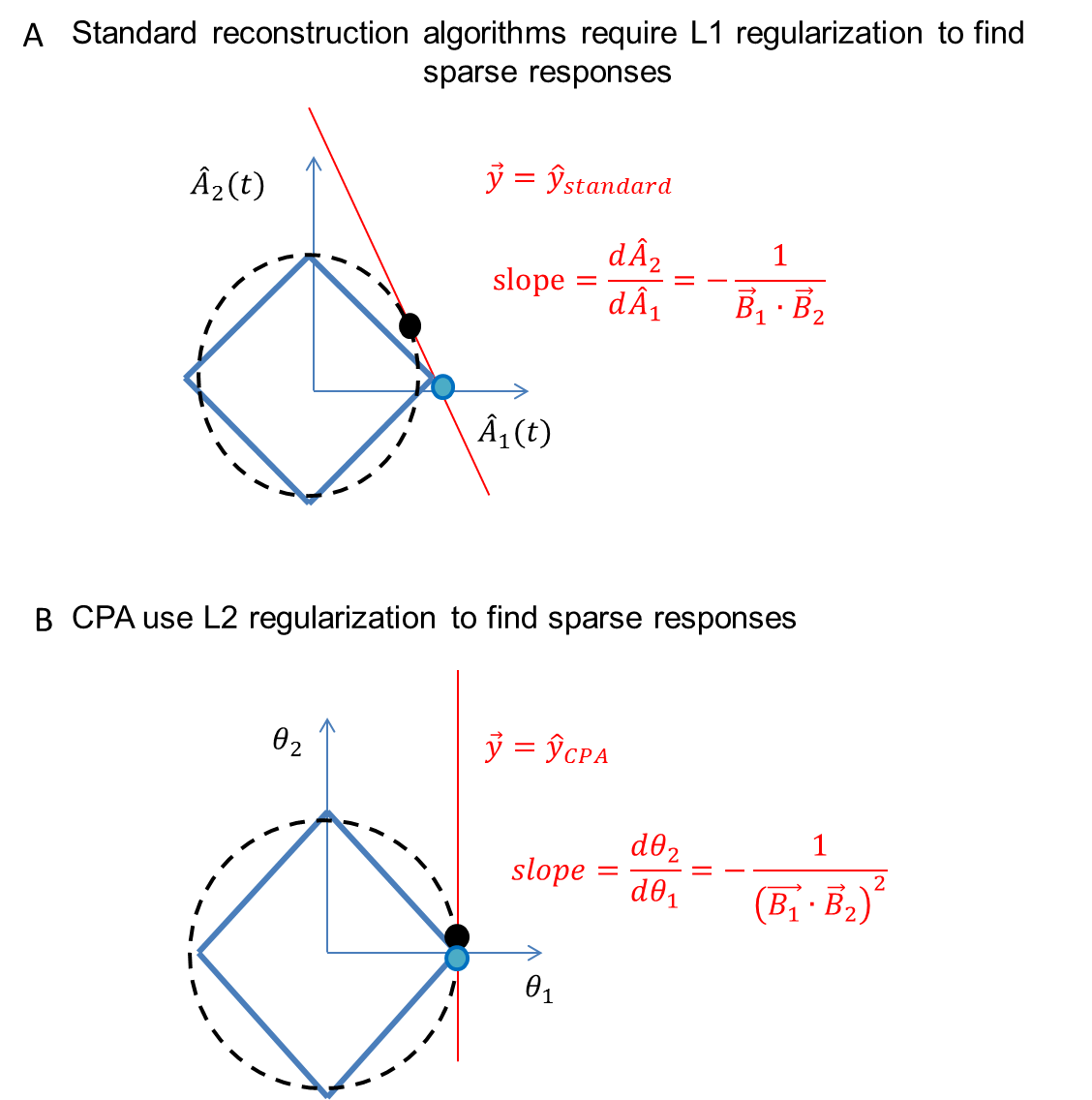}
\caption{\label{fig:1}\textbf{A} Standard methods require the L1 norm to find sparse solutions. \textbf{B} CPA change in geometry permits sparse solutions to be found using the L2 norm. }
\label{F1}
\end{figure}
How does the L2 minimization of the presence parameters  $\theta_i$ in CPA determine a sparse distribution for the dictionary atoms? The reason lies in the change in the geometry of the manifold of possible solutions (see Figure \ref{F1}). In order to gain some intuition, we will use a simple signal that is generated from a single dictionary atom yielding a single temporal observation:
\begin{equation}\label{simple_signal_descript}
\overrightarrow{y}  =  A_1 \overrightarrow{B_1}\ \text{.}
\end{equation}
We will assume a small dictionary of two atoms: $\overrightarrow{B_1}$ and $\overrightarrow{B_2}$. In standard approaches, we would estimate the contributions $\widehat{A_1}$ and $\widehat{A_2}$. All the solutions to the equation  $\overrightarrow{y}=\widehat{y}=  \widehat{A_1}\overrightarrow{B_1}+\widehat{A_2}\overrightarrow{B_2}$  describe a line in the space $\left(A_1,A_2 \right)$. The sparsest solution is the intersection of this solution line with the horizontal axis, that is:
\begin{equation}\label{sparse_solution}
\widehat{A_1}  =  A_1 \text{ and } \widehat{A_2}  =  0 \text{.}
\end{equation}
We can show that the slope of the solution line is given by:
\begin{equation}\label{slope_normal}
\frac{d\widehat{A_2}}{d\widehat{A_1}} =-\frac{1}{\overrightarrow{B_1} \cdot \overrightarrow{B_2}} \text{.}
\end{equation}
On the other hand, the family of circles
\begin{equation}\label{family_circles}
\widehat{A_1}^2+\widehat{A_2}^2=k
\end{equation}
are the loci with an equal L2 norm. The minimum L2 norm of the solutions is determined by the radius of the largest circle that is enclosed by the solution line $\overrightarrow{y(t)}=\widehat{y(t)}$, where the tangent point is the L2 minimal solution. We notice that the steeper the slope of the solution line, the closer the minimum L2 norm solution would be to the sparsest solution. Given this slope, the L2 minimal solution would be, in general, far from the sparsest solution. In contrast, the family of squares
\begin{equation}\label{family_squares}
|\widehat{A_1}|+|\widehat{A_2}|=k
\end{equation}
represents the loci with an equal L1 norm. The intersection between the largest square enclosed by the solution line corresponds to the sparsest solution.
Conversely, in the space described by the CPA presence parameters $\left(\theta_1,\theta_2 \right)$, the slope of the solution line $\overrightarrow{y(t)}=\widehat{y(t)}=\theta_1\left(y(t)\cdot\overrightarrow{B_1}\right)\overrightarrow{B_1} + 
\theta_2\left(y(t)\cdot\overrightarrow{B_2}\right)\overrightarrow{B_2}$ is: 
\begin{equation}\label{slope_cpal}
\frac{d\theta_1}{d\theta_2} =-\frac{1}{\left( \overrightarrow{B_1} \cdot \overrightarrow{B_2} \right)^2} \text.
\end{equation}
The slope of the solution line in CPA has increased by a factor of $\frac{1}{\left( \overrightarrow{B_1} \cdot \overrightarrow{B_2} \right)}$ compared to other methods. For a large enough number of dimensions $N$, the dictionary atoms are close to being orthogonal to each other, and this slope increase is very large. Having a large slope for the line $\overrightarrow{y(t)}=\widehat{y(t)}$ makes the L2 minimal solution very close to the sparsest solution, meaning that the values of presence parameters of the atoms present are much larger than the values of the presence parameters of the atoms that are not present. We can formalize this idea, and we will show that the L2 regularized solution can deal with very large dictionaries, where  $M$,  the  total number of dictionary atoms, grows exponentially with $N$, the number of dimensions. 
\begin{theorem}
Assuming that:
\begin{enumerate}
\item dictionary atoms that are simultaneously present are orthogonal to each other and, \label{cond_1}
\item the amplitude modulations of the atoms that are present are uncorrelated to each other in time, \label{cond_2}
\end{enumerate}
then the regularized L2 solution of CPA will find an average presence parameter $|\overline{\theta_i}|$ for $i \in A$ that is larger than $|\theta_k| \text{ , }\forall k \notin A $  as long as the number of dimensions $N$ is:
\begin{equation} \label{L2_CPA_condition}
  N > 4k^2 \log{\left(M\right)} \text{,}
\end{equation}
where $k$ is the number of simultaneously present dictionary atoms.  
\end{theorem}
\begin{proof}
The expected square error $\langle V\rangle$ between the observation and the presence parameters is given by:
\begin{equation}\label{temporal_error_function}
\langle V\rangle=\langle \left| \overrightarrow{y(t)}-\widehat{y(t)} \right|^2 \rangle \text{,}
\end{equation}
where the brackets $\langle \rangle$ indicate the average over time. 
The values of the presence parameters $\theta_i$ that minimize the expected square error can be calculated by taking the derivatives for all the presence parameters $\theta_i,i=1,\dots,M$:
\begin{equation}\label{derivative_over_theta}
\frac{d\langle V\rangle}{d\theta_i}=0\text{.}
\end{equation}
If we define as $A$ the set of $k$ dictionary atoms that are present in a signal, we can simplify (for details on the derivation, see \cite{Otazu2011AScenes.}) the $M$ equations defined in \ref{derivative_over_theta}, using the assumptions \ref{cond_1} and \ref{cond_2}, to the following $M$ equations:
\begin{equation}\label{cpa_system_of_equations}
\begin{cases}
0=\theta_l+\sum_{k\notin{A}}^{} {\theta_k (c_{k,l})^2 }-1\         & \quad \text{for }l \in{A} \\

0=\sum_{i\in A}^{} { c_{l,i}\langle (A_i)^2\rangle \left(\theta_i c_{l,i} -c_{l,i}+ \sum_{k\notin A}^{}{c_{i,k}\theta_k c_{k,l}} \right) }           & \quad \text{for }l \notin{A}\text{, } \\
  \end{cases}
\end{equation}
where 
\begin{equation}\label{def_ckl}
c_{k,l}=\overrightarrow{B_k} \cdot \overrightarrow{B_l} \text{.}
\end{equation}
This set of M equations is satisfied by the following solution, which identifies the atoms present in a signal:
\begin{equation} \label{solution_CPA}
\begin{cases}
\theta_i=1,& \quad \text{for } i \in{A} \\
\theta_i=0,& \quad \text{for } i \notin{A}\text{.} \\
\end{cases}
\end{equation}
However, if $\Phi{\Phi}^T$  does not have an inverse, solution \ref{solution_CPA} is just one of multiple possible solutions. We will show that the solution with the minimum L2 norm of the presence parameters identifies the sources present. 
In order to find the solution with the minimum L2 norm, we will use the Lagrange multiplier method.  We will define as a cost function the sum of the squares of the values of the presence parameters. We add as constraints the $M$ equations defined in \ref{cpa_system_of_equations},  multiplied by factors $\lambda_i, i=1\dots M$, which are the Lagrange multipliers.
The new cost function is:
\begin{equation}\label{CPA_regular_cost_function}
C=\sum_{i=1}^{M}{(\theta_i)^2} 
+\sum_{i \in A}^{}{\lambda_i \left( \theta_i+\sum_{k\notin{A}}^{} {\theta_k (c_{k,i})^2 }-1\ \right) }
+\sum_{l \notin{A}}^{} { \lambda_l  \left( \sum_{i\in A}^{} { c_{l,i}\langle (A_i)^2\rangle \left(\theta_i c_{l,i} -c_{l,i}+ \sum_{k\notin A}^{}{c_{i,k}\theta_k c_{k,l}} \right) }  \right) }\text{.}
\end{equation}
If we take the derivatives for $\theta_k \text{, } k\notin{A}$ and make them equal to zero, we obtain:
\begin{equation}\label{dc_dtheta_k}
\frac{dC}{d\theta_k}=2 \theta_k 
+\sum_{i \in A}^{}{\lambda_i   { (c_{k,i})^2 }  }
+\sum_{l \notin{A}}^{} { \lambda_l  \left( \sum_{i\in A}^{} { c_{l,i}\langle (A_i)^2\rangle  {c_{i,k} c_{k,l}}  }  \right) } =0\text{.}
\end{equation}
We can express the presence parameter $\theta_k \text{ } k\notin A$ as a function of the Lagrange multipliers,  yielding:
\begin{equation}
\theta_k =-\sum_{i \in A}^{}{\frac{\lambda_i} {2}   { (c_{k,i})^2 }  }
-\sum_{l \notin{A}}^{} { \frac{\lambda_l}{2}  \left( \sum_{i\in A}^{} { c_{l,i}\langle (A_i)^2\rangle  {c_{i,k} c_{k,l}}  }  \right) } \text{.}
\end{equation}
If we take the derivatives for $\theta_i \text{, } i\in{A}$ and make them equal to zero, we obtain:
\begin{equation}\label{dc_dtheta_in}
\frac{dC}{d\theta_i}=2 \theta_i  
+ \lambda_i 
+ \sum_{l \notin{A}}^{} { \lambda_l    { (c_{l,i})^2\langle (A_i)^2\rangle }   } =0 \text{.}
\end{equation}
We can also express the presence parameter $\theta_i \text{ } i\in A$ as a function of the Lagrange multipliers, yielding: 
\begin{equation}\label{theta_in_lagrange}
\theta_i=-\frac{\lambda_i}{2}-\sum_{l \notin{A}}^{} { \frac{\lambda_l}{2}    { (c_{l,i})^2\langle (A_i)^2\rangle }   } \text{.}
\end{equation}
The presence parameters $\theta_i$ that belong to the active set are related to the presence parameters $\theta_k$ that are not part of the active set by way of the Lagrange multipliers. We can use these relationships to infer the relative sizes of the presence parameters. 
If we calculate the sum of all $\theta_i \text{ , } i\in A$, we obtain:  
\begin{equation}\label{summ_all_theta_i}
\sum_{i \in A}^{}{\theta_i}= \sum_{i \in A}^{}{\left( -\frac{\lambda_i}{2}-\sum_{l \notin{A}}^{} { \frac{\lambda_l}{2}    { (c_{l,i})^2\langle (A_i)^2\rangle }   } \right)} \text{.}
\end{equation}
We can find an upper bound for $\theta_k \text{ , }k \notin A$, given by:
\begin{equation}
\left|\theta_k\right| =\left|-\sum_{i \in A}^{}{\frac{\lambda_i} {2}   { (c_{k,i})^2 }  }
-\sum_{l \notin{A}}^{} { \frac{\lambda_l}{2}  \left( \sum_{i\in A}^{} { c_{l,i}\langle (A_i)^2\rangle  {c_{i,k} c_{k,l}}  }  \right) }\right| \leq \left|\sum_{i \in A}^{}{\theta_i} \right|u \text{,}
\end{equation}
where $u$  is the mutual coherence\cite{Mallat1993MatchingDictionaries} of the dictionary matrix, that is:
\begin{equation}
u=\max {|c_{i,j}|}, i\not = j \text{.}
\end{equation}
We can use the average $\overline{\theta_i} $ to obtain the following inequality:
\begin{equation}
\left|\theta_k\right| \leq \left|\sum_{i \in A}^{}{\theta_i} \right|u = \left|\overline{\theta_i}\right|ku<\left|\overline{\theta_i}\right|\text{.}
\end{equation}
The last inequality holds as long as:
\begin{equation} \label{ineq_coherence}
ku<1.
\end{equation}
For a random dictionary, and for large number of dimensions $N$, the mutual coherence $u$ converges to \cite{Cai2011LIMITINGMATRICESc}:
\begin{equation}
u\rightarrow 2\frac{\sqrt[]{\log M}}{\sqrt[]{N}}.
\end{equation}
So our inequality \ref{ineq_coherence} will become:
\begin{equation}
 2k\frac{\sqrt[]{\log M}}{\sqrt[]{N}}<1.
\end{equation}
Therefore, the number of dimensions $N$ that guarantees that $|\overline{\theta_i}|$ for $i \in A$ is larger than $|\theta_k| \text{ , }\forall k \notin A $ is given by:
\begin{equation} \label{reg_cpa_limits}
N>4k^2\log{M}.
\end{equation}
\end{proof}
\begin{figure}
\centering
\includegraphics[width=0.6\textwidth]{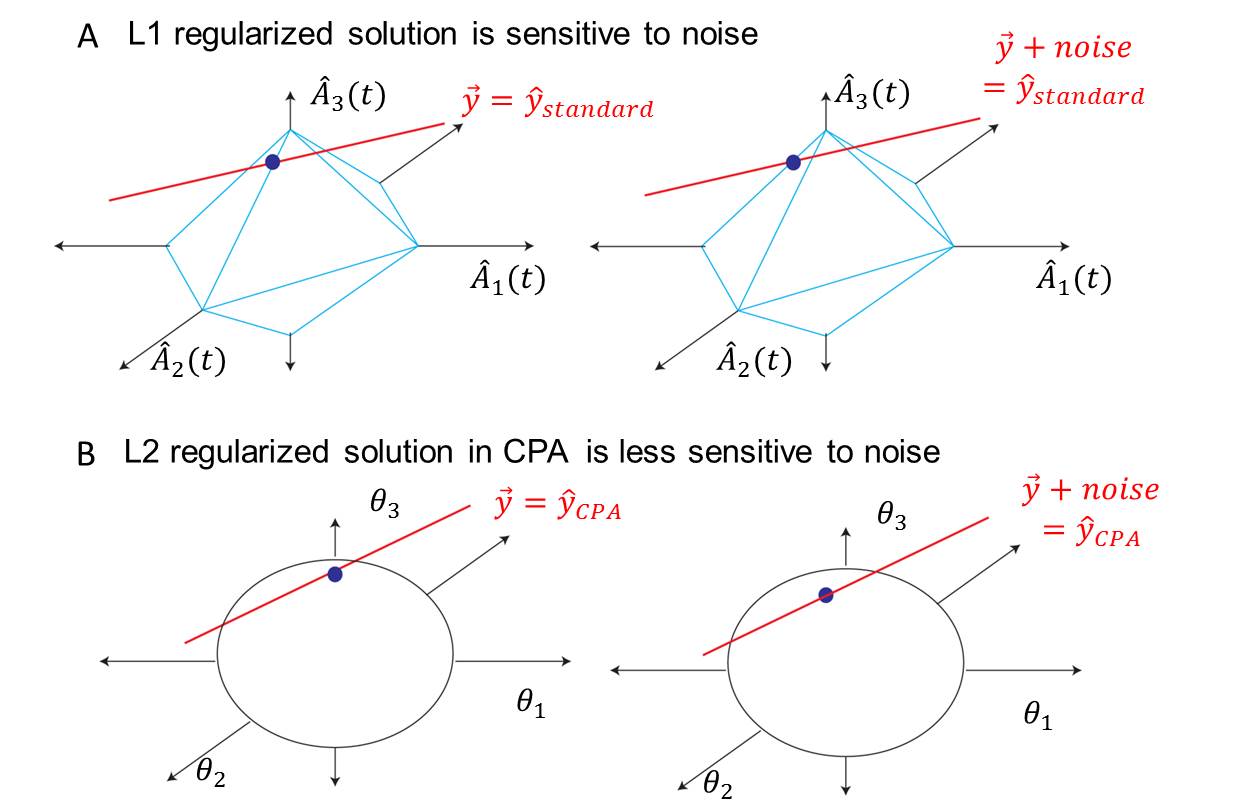}
\caption{\label{fig:1}\textbf{A} Addition of noise can cause the minimal L1 norm solution to switch to a different sparse solution set, represented by the intersection of the solutions line with a different vertex of the L1 ball. \textbf{B} Addition of noise causes less perturbation on the minimal L2 norm solution in CPA, represented by the intersection of the solutions line with the L2 ball. }
\label{F2}
\end{figure}
The number of dimensions required by the regularized CPA is increased by a factor of $4k$ compared to the limit provided by the RIP for reconstructions using the L1 regularization\cite{Candes2005DecodingProgramming}. However, the use of the L2 norm, as opposed to an L1 norm, allows an analytical expression to be found for the solution. In addition, using the smooth L2 confers CPA robustness in the presence of noise  (see Figure \ref{F2}). Intuitively, the solution of CPA using the L2 regularization is given by the intersection of a hyperplane and a hypersphere. Added noise would change the hyperplane, but given the smooth nature of the hypersphere, the new intersection would still be close to the original solution. On the other hand, for sparse representation methods, the solution is given by the intersection of a hyperplane and a high-dimension polyhedron. Perturbing the hyperplane could radically change the intersection point. While still giving a sparse solution, this would yield a different sparse set than the original solution

\section{Efficient implementation of the L2 regularized CPA}\label{iCPA and performance}\label{iCPA}
\begin{figure}
\centering
\includegraphics[width=0.6\textwidth]{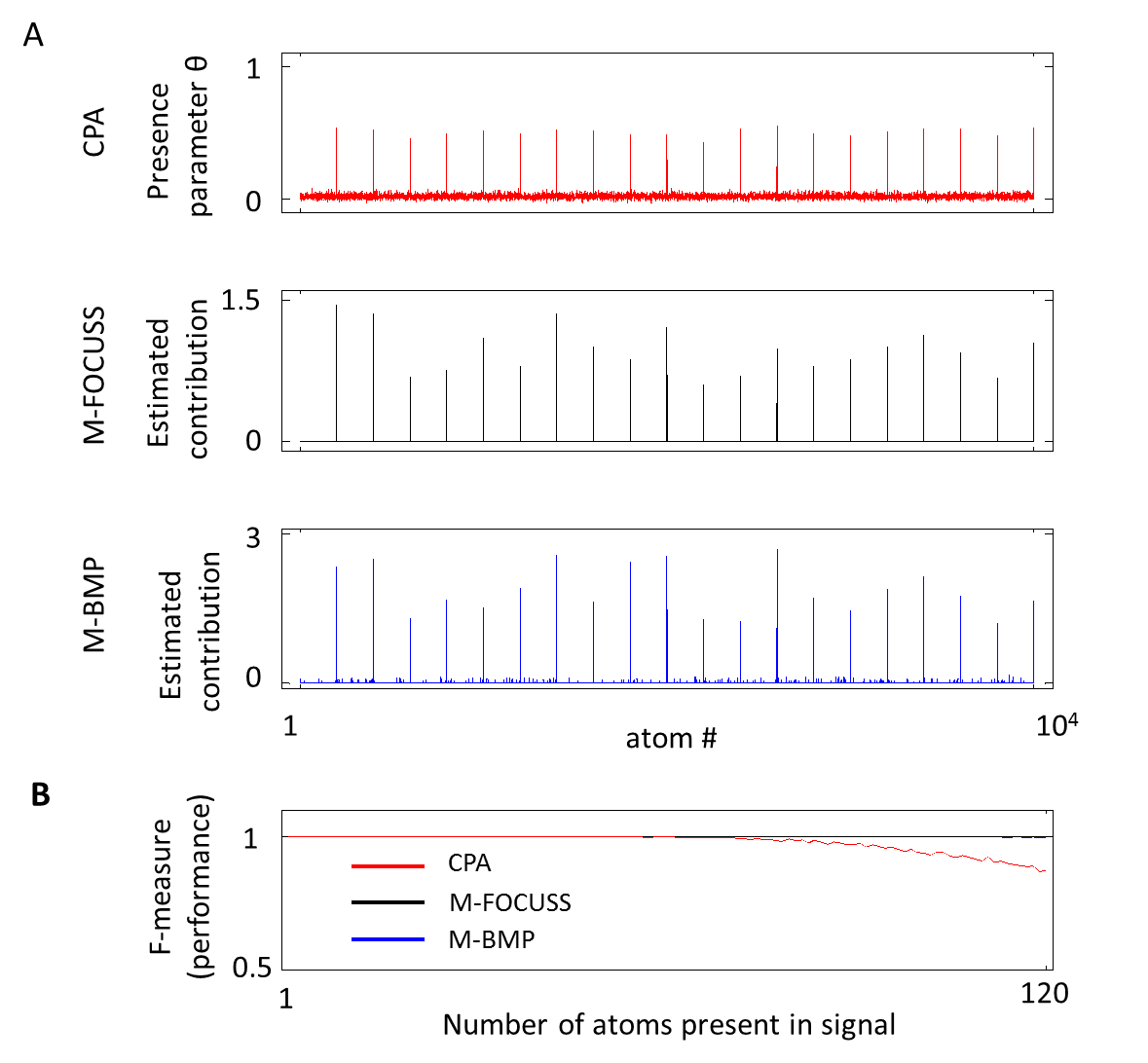}
\caption{\label{fig:1}\textbf{A} Example of the responses of the three algorithms to a signal composed of $k=20$ atoms of dimension $N=500$. The three algorithms had access to a dictionary of $M=10000$ atoms, which included the 20 atoms that generated the signal. \textbf{B} The performance of CPA for complex signals of more than 80 elements diminished compared to the performances of the other algorithms.}
\label{F3}
\end{figure}

\subsection{Iterative solution of CPA}
The CPA estimate is linear with respect to the presence parameters $\theta_i$, and it can be estimated using equation \ref{standard_CPA}. This equation requires the inversion of the $M$ by $M$ matrix $\Phi{\Phi}^T$. For a large $M$, matrix inversion is not numerically stable. It has been shown \cite{Otazu2011AScenes.} that the solution to equation \ref{standard_CPA} can be found in an iterative, numerically stable manner that is more appropriate for a real-time application. Here we will show that a similar set of iterative  equations can be used in the L2 regularized case, that is, to solve equation \ref{regularized_CPA}.

In the iterative version of CPA (or iCPA) the temporal observations $\overrightarrow{y(t)}$  from $t=1\dots T-1$ are processed to calculate a presence parameter set $\Theta(T-1)$. Upon arrival of a new observation $\overrightarrow{y(T)}$, this estimate of the presence parameters is updated to $\Theta(T)$. The update on the presence parameter set is proportional to the estimation error:
\begin{equation} \label{recursive_cpa_theta}
\Theta(T)=\Theta(T-1)+P(T) {\left( \phi(T) \right)}^T \left(\overrightarrow{y(T)}-\widehat{y(T)} \right)\text{, }
\end{equation}
where $\phi(T)$ is the $N$ by $M$ projection matrix as defined in equation \ref{projection_matrix} and calculated using the current observation $\overrightarrow{y(T)}$. The current estimate $\widehat{y(T)}$ is derived from the projection matrix $\phi(T)$  and the previous estimate of the presence parameters $\Theta(T-1)$ as follows:
\begin{equation} \label{recursive_cpa_y_estimate}
\widehat{y(T)} = \phi(T) \Theta(T-1) \text{.}
\end{equation}
The proportionality factor $P(t)$   converts the estimation error $\left(\overrightarrow{y(T)}-\widehat{y(T)} \right)$ into the update of the presence parameters as given by:
\begin{equation} \label{P_T}
P(T) = P(T-1) - P(T-1) {\left( \phi(T) \right)}^T \left( I + \phi(T) P(T-1) {\left( \phi(T) \right)}^T \right)^{-1} \phi(T)P(T-1) \text{, }
\end{equation}
where $I$ is the $N$ by $N$ identity matrix. The proportionality factor $P(t)$ is in fact  the $M$ by $M$ matrix that is calculated by matrix  inversion in equation \ref{standard_CPA}, that is:
\begin{equation}
P(T)=(\Phi{\Phi}^T)^{-1}={\left( \sum_{t=1}^{T}{{\phi(t)}^T\phi(t)}\right)}^{-1} \text{.}
\end{equation}
Equation \ref{P_T} defines an iterative relation where the new $P(T)$ is calculated using the previous value $P(T-1)$ combined with the new value of the observation  $\overrightarrow{y(T)}$ through the projections $\phi(T)$. The real-time implementation is computationally more stable as it requires the inversion of an  $N$ by $N$ matrix as opposed to an $M$ by $M$ matrix $(M>>N)$.

\subsection{Iterative solution of CPA with L2 regularization}
In the case of the L2 regularized CPA, we need to find the inverse of $(\Phi{\Phi}^T +\lambda I )$ to calculate the presence parameters using equation \ref{regularized_CPA}. We will show that we could also use the iterative equation \ref{P_T} to calculate this inverse in an efficient manner.
The L2 regularized version of $P(T)$  is:
\begin{equation}
P_{regularized}(T)=(\Phi{\Phi}^T+\lambda I)^{-1}={\left( \sum_{t=1}^{T}{{\phi(t)}^T\phi(t)}+\lambda I \right)}^{-1} \text{.}
\end{equation}
At T=0, the estimate of $P_{regularized}(T)$ becomes:
\begin{equation}
P_{regularized}(0)={\left(\lambda I \right)}^{-1}=\frac{1}{\lambda}I  \text{.}
\end{equation}
Therefore, if we initialize $P(0)$  to an $N$ by $N$ identity matrix times $\frac{1}{\lambda}$ ,  the iterative equation \ref{P_T} would  calculate  the  L2 regularized solution  with  $\lambda$ as the regularization constant . By having a large  initialization value of $P(0)$, we would prefer an exact reconstruction over a sparse representation. Using this initialization and using equations \ref{recursive_cpa_theta}, \ref{recursive_cpa_y_estimate}, and \ref{P_T}  would produce the L2 regularized solution for CPA in a computationally efficient manner.

\section{CPA outperforms other sparse representation algorithms in the presence of strong novel atoms}\label{CPA_performance}
CPA can identify the atoms present in a signal, and its performance is comparable to other sparse representation algorithms (see Figure \ref{F3}). We compared CPA against two other MMV methods \cite{Cotter2005SparseVectors}, basic matching pursuit (M-BMP), a greedy algorithm that approximates the L1 regularization,  and M-FOCUSS. When a  large number of dictionary atoms were simultaneously present, CPA performance was lower than the performance of the other sparse representation algorithms. Interestingly, the performance was better than what we would have expected given the number of dimensions in equation \ref{reg_cpa_limits} , indicating that CPA has an even better performance than our estimate.

The main advantage of CPA was in the handling of novel atoms. CPA did not produce a sparse representation for a novel atom (see Figure \ref{F4}). Instead, it represented the novel atom with a dense set of presence parameters of small amplitude. The amplitude of this dense set of presence parameters was not very sensitive to the input signal amplitude, keeping them small, even as the contribution of the novel atom increased. This dense representation was different from the sparse representation of a signal composed of known atoms. 
\begin{figure}
\centering
\includegraphics[width=0.6\textwidth]{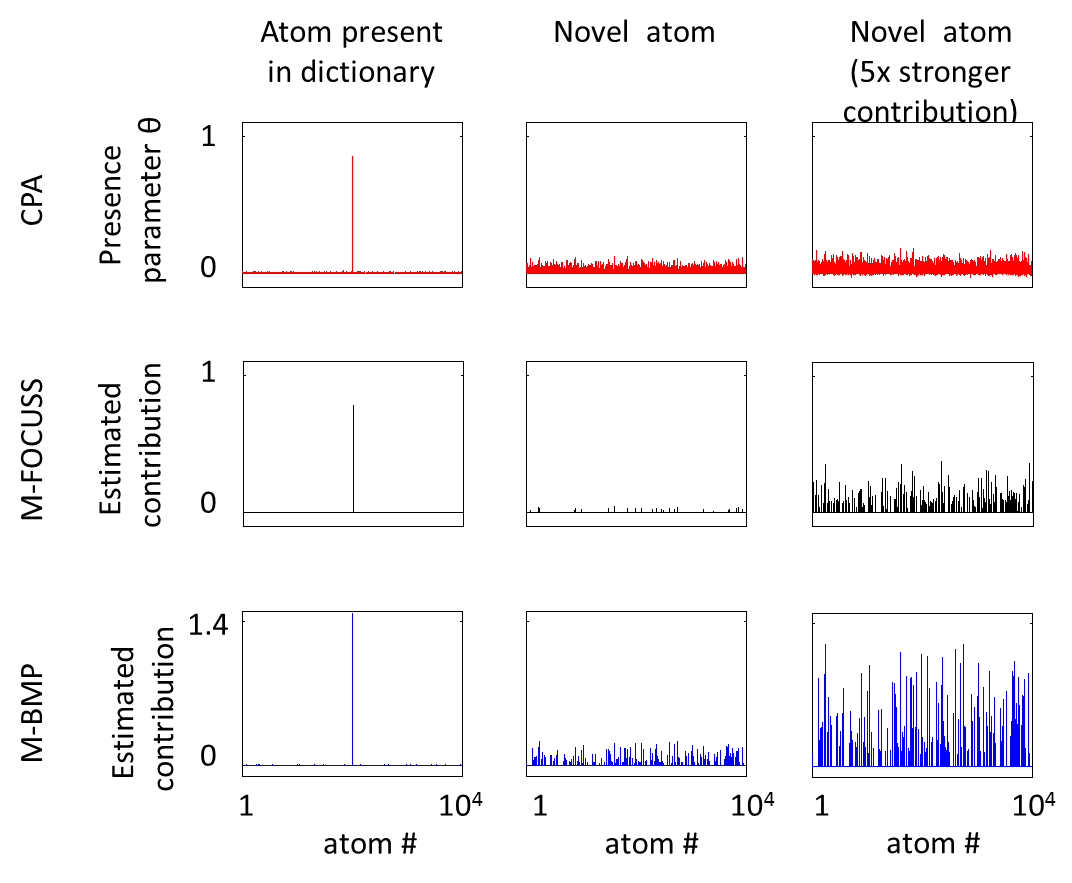}
\caption{\label{fig:1} CPA and the other algorithms generated sparse representations for an atom that belonged to the dictionary (left column). CPA representation of a novel atom was denser and more amplitude-invariant compared to other algorithms (left and right columns).}
\label{F4}
\end{figure}
\begin{figure}
\centering
\includegraphics[width=0.6\textwidth]{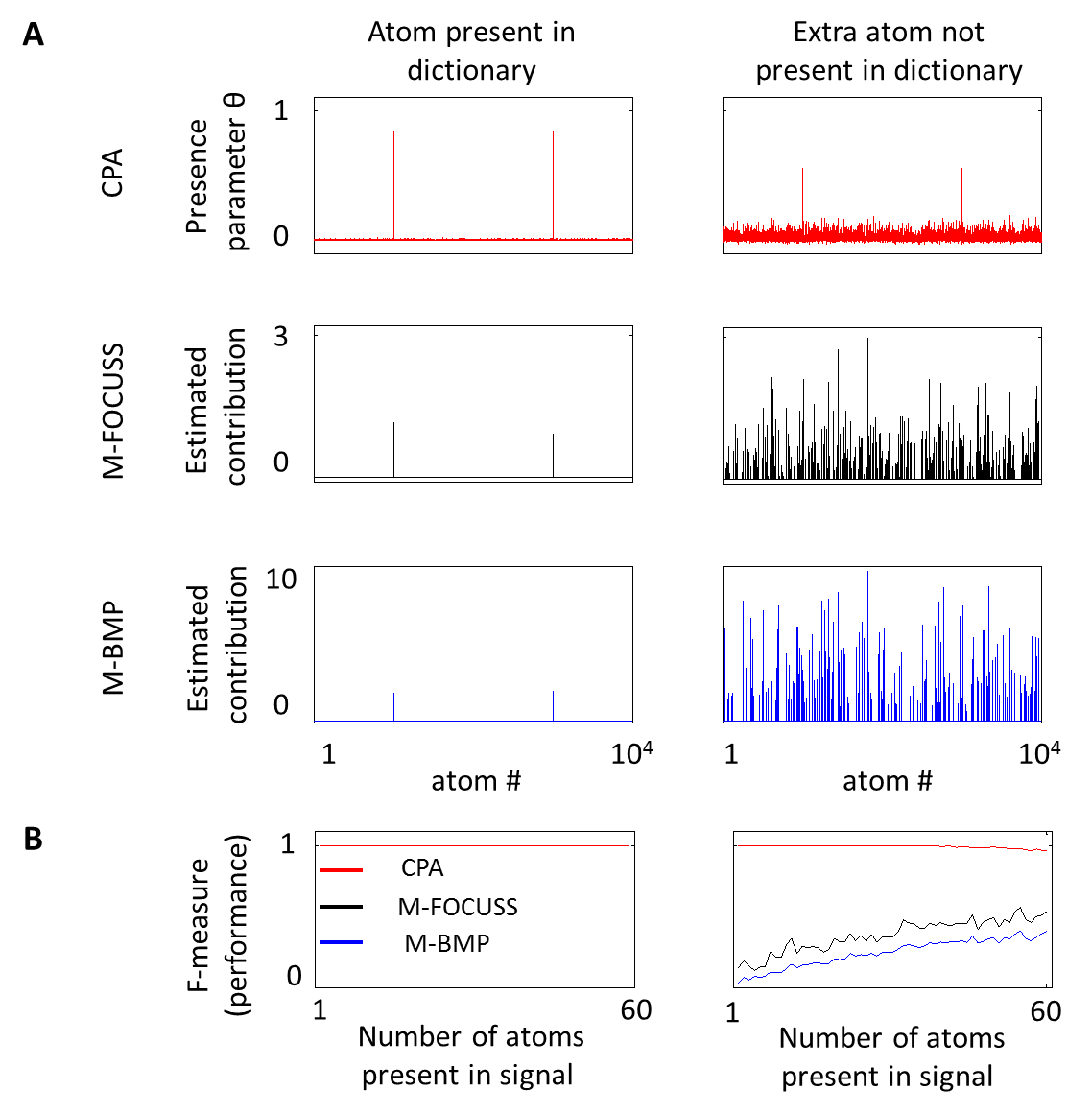}
\caption{\label{fig:1}\textbf{A} Example of the responses of the three algorithms to a signal composed of $k=2$ atoms from the dictionary, in the absence (left column) and presence (right column)  of a novel atom with a strong contribution. CPA presence parameters for the two atoms from the dictionary were very salient, whereas for the other two algorithms, the two atoms were effectively masked by the novel atom contribution. \textbf{B} The performance of CPA in detecting atoms that were in the dictionary ($k=1\dots60$) was comparable to that of the other algorithms in the absence of strong novel atoms (left column). CPA performance for complex signals was more robust to a novel atom than the other algorithms (right column).   }
\label{F5}
\end{figure}
M-FOCUSS and M-BMP did represent a novel atom with a sparse representation whose magnitude depended on the contribution of the novel atom to the signal. This sparse representation of a novel atom was indistinguishable from a representation of a signal composed of multiple known atoms.

In a complex scene that consisted of a strong novel atom and other atoms that belonged to the dictionary, CPA identified the dictionary atoms, with the presence of the strong novel atom adding only small amounts of noise distributed uniformly across the presence parameters (see Figure \ref{F5}). In contrast, the representations produced by the sparse dictionary algorithms were dominated by the novel atom representation, which masked the known atoms. We explored several parameters for the regularization constant $\lambda$ of M-FOCUSS (see Figure \ref{F6}), but we could not find a parameter regime that increased the performance, indicating that the deleterious effects of strong novel atoms could not be overcome with parameter adjustment but are intrinsic to the L1 regularization approach.
\begin{figure}
\centering
\includegraphics[width=0.6\textwidth]{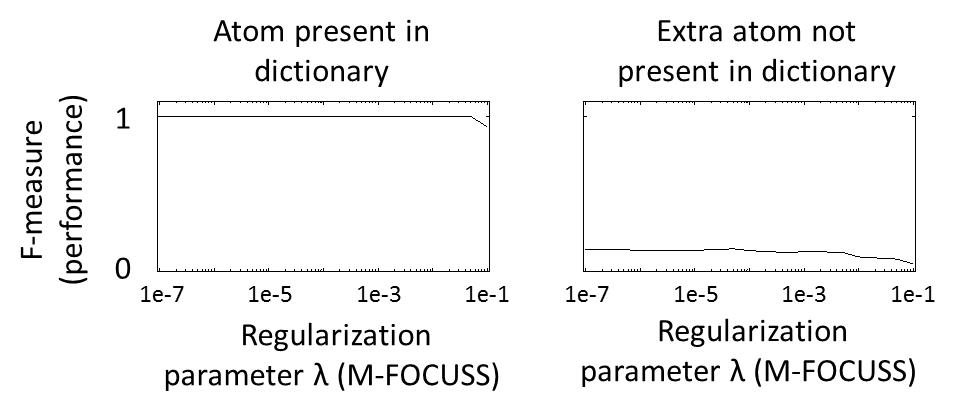}
\caption{\label{fig:1} Performance of M-FOCUSS for detecting atoms in a signal composed of $k=2$ atoms as a function of the regularization parameter $\lambda$ in the absence (left column) and presence (right column) of a novel atom. The performance of M-FOCUSS cannot be improved by changing $\lambda$ over six orders of magnitude.  }
\label{F6}
\end{figure}

\section{Discussion}
We have extended CPA, a biologically inspired algorithm,  for the case of overcomplete dictionaries. We found that by changing the geometry of the solution space, CPA found sparse solutions using the L2 norm. We presented a Kalman filter implementation of the L2 regularized solution for a numerical stable implementation of this MMV algorithm. CPA outperformed other sparse representation algorithms in identifying sources in the presence of strong novel atoms. 

CPA is particularly suitable when we are interested in identifying previously acquired independent atoms that might be present in a signal but there might also be novel atoms that are not part of the dictionary masking them.  This makes CPA suitable for online applications where the currently used dictionary does not contain all the possible atoms that may appear in a signal and we are still learning the dictionary. In contrast, algorithms that specifically minimize the L1 norm will find a sparse solution even if the atom generating the observed signal is not part of the dictionary. In this case the estimated amplitudes $A_i(t)$ would be temporally correlated in time and would not constitute independent sources. 

We found that CPA is computationally more expensive than the other algorithms tested. Evaluation of $P(t)$, using equation \ref{P_T}, is the most computationally expensive calculation, requiring the multiplication of 3 $M$ by $M$ matrices, that is, CPA is $\mathcal{O}(M^3)$. However, all operations for CPA are matrix multiplications and CPA could be optimized for implementation using GPUs.
 
CPA does not produce a sparse representation for novel atoms.  The lack of a sparse representation in the presence of an input might be used as an indication that a new atom should be incorporated into the dictionary.

\section{Appendix}
\subsection{Dictionary atoms}
All simulations were performed using Matlab. For all simulations we used a dictionary of $M=10000$ atoms. The dimensions of the atom were $N=500$. Initially, the atoms in the dictionary were independently generated from a Gaussian distribution of zero mean and a variance of 1. Each dictionary atom was normalized such as $\overrightarrow{B_k}\cdot\overrightarrow{B_k}=1, k=1\dots M$. An identical procedure was used to generate the novel atoms that were not part of the dictionary.
\subsection{Signal generation}
The temporal-varying amplitudes for the atoms present, $A_k(t)$, were taken independently from a Gaussian distribution of zero mean and a standard deviation of 1. We used $T=10$ (total number of observed samples of the signal). The signal generated by combining the dictionary atoms was corrupted by additive Gaussian noise of standard deviation $1/10$ of the standard deviation of the atom-generated signal.  The amplitude of the novel atom in Figures \ref{F5} and \ref{F6} was taken from a Gaussian distribution of zero mean and a standard deviation of 10. 
\subsection{Corrected Projections Algorithm}
We implemented the iterative version of CPA using equations \ref{recursive_cpa_theta}, \ref{recursive_cpa_y_estimate}, and \ref{P_T} . We report the presence parameter after the last sample $T$ is processed. For all the simulations, we used a value of $\lambda=1/2.5$ for the regularization constant. The code for implementing CPA is available \footnote[1]{https://github.com/gotazu/CPA}. 
\subsection{M-FOCUSS}
We used the regularized M-FOCUSS as described in \cite{Cotter2005SparseVectors}, and we used the code in  \footnote[2]{https://sites.google.com/site/researchbyzhang/software}. We used a regularization constant value $\lambda=1e-3$, $p_{norm}=0.8$, the threshold for stopping iteration $\epsilon=1e-8$, the threshold for pruning small gamma, prune $\gamma=1e-4$, and the maximum number of iterations was set to 500. We used the same set of parameters for all simulations, except for Figure \ref{F6}, where we systematically changed the regularization constant value $\lambda$ between $1e-7$ and $1e-1$.
\subsection{M-BMP}
We used the M-BMP as described in \cite{Cotter2005SparseVectors}. We selected 200 as the maximal number of iterations. The same set of parameters were used for all the simulations.

\subsection{Algorithms performance assessment}
In order to evaluate the performance of the algorithms we used the F-measurement, a more appropriate measurement for sparse representations than ROC analysis. It is defined as 
$F=2 \frac{precision∙recall}{precision+recall}$.
Precision is the fraction of detected atoms that were actually present in the signal. Recall is the fraction of atoms present in the signal that were detected. A value of $F=1$ indicates perfect detection, i.e., that all the atoms present were detected, and only those atoms were detected. For each of the three algorithms tested, CPA, M-FOCUSS, and M-BMP, we calculated a detection threshold for their output that maximized the F-measurement. For each condition, we repeated the simulation 10 times and reported the average of the optimal F-measurement.

\bibliographystyle{alpha}
\bibliography{Mendeley.bib}

\end{document}